\documentclass[12pt, twoside, leqno]{article}

\usepackage{amsmath,amsthm}
\usepackage{amssymb}
\usepackage[all,cmtip]{xy}

\usepackage{enumerate}

\usepackage{graphicx}

\newtheorem{thm}{Theorem}[section]
\newtheorem{cor}[thm]{Corollary}
\newtheorem{lem}[thm]{Lemma}
\newtheorem{prop}[thm]{Proposition}

\newtheorem{question}[thm]{Question}

\theoremstyle{definition}
\newtheorem{defin}[thm]{Definition}
\newtheorem{rem}[thm]{Remark}
\newtheorem{exa}[thm]{Example}

\numberwithin{equation}{section}

\newcommand{\fp}{\mathfrak{p}}
\newcommand{\G}{\Gamma}

\newcommand{\Hom}{\text{Hom}}

\newcommand{\af}{\text{af}}

\newcommand{\Sp}{\text{Spec} \,}
\newcommand{\cok}{\text{coker}}

\begin{document}

%%%%% To ease editing, for IMPAN journals add:

\baselineskip=17pt

%%%%%%%%%%%%%%%%

\title{Between the genus and the $\G$-genus of an integral quadratic $\G$-form}

\author{Rony A. Bitan\\
Department of Mathematics, Bar-Ilan University\\ 
Ramat-Gan 5290002, ISRAEL\\
E-mail: rony.bitan@gmail.com}

\date{}

\maketitle

%    Some definitions useful in producing this sort of documentation:
%\chardef\bslash=`\\ % p. 424, TeXbook
%    Normalized (nonbold, nonitalic) tt font, to avoid font
%    substitution warning messages if tt is used inside section
%    headings and other places where odd font combinations might
%    result.
\newcommand{\ntt}{\normalfont\ttfamily}
%    command name
\newcommand{\cn}[1]{{\protect\ntt\bslash#1}}
%    LaTeX package name
\newcommand{\pkg}[1]{{\protect\ntt#1}}
%    File name
\newcommand{\fn}[1]{{\protect\ntt#1}}
%\newcommand{\fn}[1]{{\protect\ntt#1}}
%    environment name
\newcommand{\env}[1]{{\protect\ntt#1}}
\hfuzz1pc % Don't bother to report overfull boxes if overage is < 1pc

%       Theorem environments

%% \theoremstyle{plain} %% This is the default

%\newtheorem{theorem}{Theorem}[section]
%\newtheorem{thm}[theorem]{Theorem}
%\newtheorem{prop}[theorem]{Proposition}
%\newtheorem{lem}[theorem]{Lemma}
%\newtheorem{cor}[theorem]{Corollary}
%\newtheorem{red}[theorem]{Reduction}

%\newtheorem{conjecture}[theorem]{Conjecture}
%\newtheorem{problem}[theorem]{Problem}
%\newtheorem{claim}[theorem]{Claim}
%\newtheorem{problems}[theorem]{Problems}
%\newtheorem{questions}[theorem]{Questions}
%
%
%\theoremstyle{definition}
%\newtheorem{defn}[theorem]{Definition}  % was Def..
%\newtheorem{defns}[theorem]{Definitions}
%\newtheorem{remark}[theorem]{Remark}
%%\newtheorem{question}[theorem]{Question}
%\newtheorem{example}[theorem]{Example}
%\newtheorem{reduction}[theorem]{Reduction}
%
%\numberwithin{equation}{section}
%
%\newtheorem*{maintheorem*}{Main Theorem}
%\newtheorem{definition}{Definition}
%
%
%\newtheorem*{Notation}{Notation}
%
%\newtheorem{notation}[theorem]{Notation}

%       Math definitions

\newcommand{\surj}{\twoheadrightarrow}

\newcommand{\CE}{\mathcal{E}}
\newcommand{\CG}{\mathcal{G}}

\newcommand{\CO}{\mathcal{O}}
\newcommand{\OV}{\un{\textbf{O}}_V}
\newcommand{\SOV}{\un{\textbf{SO}}_V}

\newcommand{\ClS}{\text{Cl}_S}
\newcommand{\Cl}{\text{Cl}}

\newcommand{\CS}{\CO_S}
\newcommand{\Oiy}{\mathcal{O}_{\{\iy\}}}
\newcommand{\B}{\mathcal{B}}
\renewcommand{\k}{\varkappa}
\newcommand{\CR}{\mathcal{R}}
\newcommand{\X}{\mathcal{X}}
\newcommand{\fX}{\mathfrak{X}}
\newcommand{\wt}{\widetilde}
\newcommand{\wh}{\widehat}
\newcommand{\mk}{\medskip}
\renewcommand{\sectionmark}[1]{}
\renewcommand{\Im}{\operatorname{Im}}
\renewcommand{\Re}{\operatorname{Re}}
\newcommand{\la}{\langle}
\newcommand{\ra}{\rangle}
\newcommand{\N}{\un{N}}
\newcommand{\sh}{\text{sh}}
\newcommand{\tor}{\text{tor}}
\newcommand{\op}{\text{op}}
\newcommand{\diag}{\text{diag}}
\newcommand{\Tr}{\text{Tr}}
\newcommand{\Br}{{\mathrm{Br}}}

\renewcommand{\th}{\theta}
\newcommand{\ve}{\varepsilon}
\newcommand{\iy}{\infty}
\newcommand{\iintl}{\iint\limits}
\newcommand{\cupl}{\operatornamewithlimits{\bigcup}\limits}
\newcommand{\suml}{\sum\limits}
\newcommand{\ord}{\operatorname{ord}}
\newcommand{\bk}{\bigskip}
\newcommand{\fc}{\frac}
\newcommand{\g}{\gamma}
\newcommand{\be}{\beta}
\newcommand{\s}{\sigma}
\renewcommand{\sc}{\text{sc}}
\renewcommand{\ss}{\text{ss}}
\newcommand{\Pic}{\text{Pic~}}

\newcommand{\F}{\mathcal{F}}
\newcommand{\C}{\textbf{C}}

\newcommand{\dl}{\delta}
\newcommand{\Dl}{\Delta}
\newcommand{\lm}{\lambda}
\newcommand{\Lm}{\Lambda}
\newcommand{\om}{\omega}
\newcommand{\Om}{\Omega}
\newcommand{\ov}{\overline}
\newcommand{\un}{\underline}
\newcommand{\vp}{\varphi}
\newcommand{\BC}{\mathbb{C}}
\newcommand{\BG}{\mathbb{G}}
\newcommand{\BF}{\mathbb{F}}
\newcommand{\BP}{\mathbb{P}}
\newcommand{\bP}{\textbf{P}}
\newcommand{\BQ}{\mathbb{Q}}
\newcommand{\BM}{\Bbb{M}}
\newcommand{\BR}{\Bbb{R}}
\newcommand{\BN}{\Bbb{N}}
\newcommand{\fG}{\frak{G}}
\newcommand{\fC}{\frak{C}}
\newcommand{\CH}{\mathcal{H}}
\newcommand{\T}{\mathcal{T}}
\newcommand{\lT}{\T^{\text{lft}}}
\newcommand{\fT}{\T^{\text{ft}}}
\newcommand{\rlT}{\overline{\T}^{\text{lft}}}
\newcommand{\rfT}{\overline{\T}^{\text{ft}}}
\newcommand{\rcT}{\overline{\T}^{\text{0}}}
\newcommand{\TP}{\mathcal{P}}
\newcommand{\rX}{\overline{X}}
\newcommand{\rcX}{\overline{X}^0}
\newcommand{\js}{j_*^{\text{st}}}
\newcommand{\tmod}{\text{mod}}
\newcommand{\id}{\operatorname{id}}
\newcommand{\ep}{\epsilon}
\newcommand{\tp}{\tilde\partial}
\newcommand{\doe}{\overset{\text{def}}{=}}
\newcommand{\loc} {\operatorname{loc}}
\newcommand{\de}{\partial}
\newcommand{\z}{\zeta}
\renewcommand{\a}{\alpha}
\renewcommand{\b}{\beta}

\newcommand{\et}{\text{\'et}}
\newcommand{\fl}{\text{fl}}
\newcommand{\p}{\varphi}

\newcommand{\Z}{\mathbb{Z}}
\newcommand{\Q}{\mathbb{Q}}
\newcommand{\A}{\mathbb{A}}

%\begin{document}

%\date{}

%\address{Department of Mathematics, Bar-Ilan University, Ramat Gan 5290002, ISRAEL}
%\email{rony.bitan@gmail.com}
%\address{ R. Bitan, Department of
%Mathematics, Technion-Israel Institute of Technology -- Technion City, Haifa 32000, Israel}
%\email{rony.bitan@gmail.com}

%\baselineskip 20pt
%\setcounter{equation}{0}
%\pagestyle{plain}
%\pagenumbering{arabic}
%
%%\title{Between the genus and the $\G$-genus of an integral quadratic $\G$-form}
%\author{Rony A. Bitan}
%\thanks{This work was supported by a Chateaubriand Fellowship of the Embassy of France in Israel, 2016}
%
%
%\maketitle{}

\renewcommand{\thefootnote}{}



\renewcommand{\thefootnote}{\arabic{footnote}}
\setcounter{footnote}{0}

\begin{abstract}
%-------------------
Let $\G$ be a finite group and $(V,q)$ be a regular quadratic $\G$-form defined over an integral domain $\CS$  
of a global function field (of odd characteristic). 
We use flat cohomology to classify the quadratic $\G$-forms defined over $\CS$ 
that are locally $\G$-isomorphic for the flat topology to $(V,q)$  
and compare between the genus $c(q)$ and the $\G$-genus $c_\G(q)$ of $q$. 
We show that $c_\G(q)$ should not inject in $c(q)$. 
The suggested obstruction arises from the failure of the Witt cancellation theorem for $\CS$.  
\end{abstract}

\section{Introduction}
%--------------------------------------------------------
Let $C$ be a geometrically connected and smooth projective curve defined over the finite field $\BF_q$ ($q$ is odd).  
Let $K=\mathbb{F}_{q}(C)$ be its function field and let $\Omega$ denote the set of all closed points of $C$. 
%and let $K=\BF_q(C)$ be its global function field. % and $k^s$ a separable closure of it.   
For any point $\fp \in \Om$ let $v_\fp$ be the induced discrete valuation on $K$,    
$\hat{\CO}_\fp$ the complete discrete valuation ring with respect to $v_\fp$   
and $\hat{K}_\fp$ its fraction field.  
Any \emph{Hasse set} of $K$ namely, a non-empty finite set $S \subset \Om$,  
gives rise to an integral domain of $K$ called a \emph{Hasse domain}: 
$$ \CS := \{x \in K: v_\fp(x) \geq 0 \ \forall \fp \notin S\}. $$ 
Any $\CS$-scheme is underlined, being omitted in the notation of its generic fiber.  

\bk

Let $\G$ be a finite group, faithfully represented by $\rho:\G \hookrightarrow \textbf{GL}(V)$ 
where $V$ is a projective $\CS$-module of rank $n \geq 1$.  
We briefly write $\g$ instead of $\rho(\g)$ when no confusion may occur 
and assume $|\G|$ is prime to $\text{char}(K)$. 
Then $\G$ acts on $\textbf{GL}(V)$ on the left by conjugation: 
$$ \forall \g \in \G, A \in \textbf{GL}(V): \ \ ^\g A = \g A \g^{-1}. $$

Let $V$ be equipped with a degree two homogeneous $\CS$-form $q:V \to \CS$, 
turning $(V,q)$ into an \emph{integral} quadratic $\CS$-space,   
represented by a bilinear map $B_q: V \times V \to \CS$ such that: 
$$ B_q(u,v) = q(u+v)-q(u)-q(v). $$
We assume $(V,q)$ is $\CS$-\emph{regular}, i.e., that 
the induced homomorphism $V \to V^\vee := \Hom(V,\CS)$ is an isomorphism.   
We say it is \emph{rationally isotropic} (or just \emph{isotropic)}, 
if there exists a non-zero $v \in V$ for which $q(v)=0$.  
It is considered a \emph{$\G$-form} if it satisfies $q \circ \g = q$ for all $\g \in \G$. 
Two integral forms $(V,q)$ and $(V',q')$ are said to be $R$-\emph{isomorphic} where $R$ is an extension of $\CS$ 
if there exists an $R$-isomorphism $A: V' \cong V$ such that $q \circ A = q'$. $A$ is called an $R$-\emph{isometry}. 
Two $\G$-forms are said to be $\G$-\emph{isomorphic} over $R$ 
if there exists an $R$-isometry between them which is $\G$-equivariant.  

\bk

The classification of quadratic forms (without a group action) defined over an integral domain of a function field 
was initially studied by L.~Gerstein (\cite{Ger}) and J.\ S.\ Hsia (\cite{Hsia}) in the late seventies. 
Later on, J.\ Morales proved in \cite{Mor} that there are only finitely many $\G$-isomorphism classes of $\G$-forms defined over $\Z$ of a given discriminant. 
He also showed that the classical Hasse-Minkowski theorem, 
stating that two forms defined over a global field $K$ are $K$-isomorphic  
if and only if they are $\hat{K}_v$-isomorphic at any place $v$, 
does not hold for $\G$-forms with $\G$-isomorphisms when $K=\BQ$. 
Recently, E.\ Bayer-Fluckiger, N.\ Bhaskhar and R.\ Parimala showed in \cite{BBP}  
that this principle does hold, however, for $\G$-forms when $K$ is a global function field. 

\bk

The failure of this \emph{local-global principle} in the case of integral forms (even without a group action) is measured by the \emph{genus} of such a form. 
This term and its generalization to integral $\G$-forms are defined as follows:  

\begin{defin} 
The \emph{genus} $c(q)$ of an $\CS$-form $q$ is the set of isomorphism classes of $\CS$-forms   
that are $K$ and $\hat{\CO}_\fp$-isomorphic to $q$ for any prime $\fp \notin S$.  \\
The $\G$-\emph{genus} $c_\G(q)$ of a $\G$-form $q$ defined over $\CS$ is the set of $\G$-isomorphism classes of $\G$-forms defined over $\CS$   
that are $K$ and $\hat{\CO}_\fp$-isomorphic to $q$ for any prime $\fp \notin S$, and these isomorphisms are $\G$-ones. \\
We denote by $c^+(q)$ and $c_\G^+(q)$ the genus and the $\G$-genus of $q$ respectively, 
with respect to proper (i.e., of $\det=1$) isomorphisms only. 
\end{defin}  

This paper was motivated by the following question, as was posed to me by B. Kunyavski\u\i : 

\begin{question} \label{question}
Suppose two integral $\G$-forms share the same $\G$-genus 
and they are $\CS$-isomorphic (the $\G$-action is forgotten). 
Would they necessarily be also $\G$-isomorphic ? 
\end{question}

Any integral $\G$-form representing a class in $c_\G(q)$ clearly represents a class in $c(q)$ as well. 
So the map $\psi:c_\G(q) \to c(q)$ is well-defined,        
and we may rephrase Question \ref{question} as follows: \\
Is $\psi$ always an injection ? 
Because if the answer is no, and only then, this would mean that there exist two integral $\G$-forms  
representing two distinct classes in $c_\G(q)$, though being $\CS$-isomorphic.  

\bk

After showing in Section \ref{Classification} that [$H^1_\fl(\CS,\SOV^\G)$] $H^1_\fl(\CS,\OV^\G)$  
[properly] classifies the integral $\G$-forms that are locally $\G$-isomorphic to $(V,q)$, 
where [$\SOV^\G$] $\OV^\G$ stands for the [special] orthogonal group of $(V,q)$,  
we compare in Section \ref{genus} between the genus and the $\G$-genus of $(V,q)$, 
and give a necessary and sufficient condition for the positive answer to Question \ref{question}.  
In order to provide a counter-example, we refer in Section \ref{Obstruction} more concretely 
to the case in which  $(\OV^\G)^0$ 
is the special orthogonal group of another isotropic form $(V',q')$. 
Based on a result established in \cite{Bit1} and \cite{Bit2}, stating that if $q$ is isotropic of rank $\geq 3$, 
then $c(q) \cong \Pic(\CS)/2$, we show that if $q'$ is isotropic of rank $2$ %and $-1 \in (\BF_q^\times)^2$,   
then it may possess twisted forms which are only \emph{stably isomorphic}, 
namely, that become isomorphic over (any) regular non-trivial extension of $V'$. 
So our suggested obstruction to Question \ref{question} 
arises from the failure of the Witt cancellation theorem over $\CS$.     

\bk

{\bf Acknowledgements:} 
I would like to thank my Post Doc. host in Camille Jordan Institute in the University Lyon1, P.~Gille, for his useful advice and support, 
and B.~Kunyavski\u\i\ for valuable discussions concerning the topics of the present article.

\bk

\section{The classification via flat cohomology} \label{Classification}
%-----------------------------------------------
The following general framework that appears in \cite[\S 2.2.4]{CF} allows us to derive some known facts  
about the classification of integral forms via flat cohomology, to integral $\G$-forms.   

\begin{prop} \label{flat classification}
Let $R$ be a scheme and $X_0$ be an $R$-form, namely, 
an object of a fibered category of schemes defined over $R$. 
Let $\textbf{Aut}_{X_0}$ be its $R$-group of automorphisms. 
Let $\mathfrak{Forms}(X_0)$ be the category of $R$-forms that are locally isomorphic for some topology to $X_0$   
and let $\mathfrak{Tors}(\text{Aut}_{X_0})$ be the category of $\text{Aut}_{X_0}$-torsors in that topology.    
The functor 
$$ \p:\mathfrak{Forms}(X_0) \to \mathfrak{Tors}(\textbf{Aut}_{X_0}): \ X \mapsto \textbf{Iso}_{X_0,X} $$ 
is an equivalence of fibered categories.     
\end{prop}

We first implement this Proposition on split torsion $R$-groups.   
For a non-negative integer $m$ we consider the $R$-group $\un{\mu}_m := \Sp R[t]/(t^m-1)$. 
The pointed-set $H^1_\fl(R,\un{\mu}_m)$ classifies $\un{\mu}_m$-torsors, 
namely, $R$-groups that are locally isomorphic to $\un{\mu}_m$ for the flat topology.   
We briefly introduce another description of these elements, 
as can be found for example in \cite[\S 5.1]{AG}.   

\bk

An $m$-degree $R$-\emph{Kummer pair} is a couple $\Lm = (L,h)$ 
consisting of a rank $1$ projective $R$-module $L$  
and an isomorphism $h:R \cong L^{\otimes m}$.  
It gives rise to a $\un{\mu}_m$-torsor $E_\Lm$  
assigning to any extension $R'/R$ the group: 
$$ E_\Lm(R') = \{\p \in L^\vee \otimes R' \ : \ \p^{\otimes m} = h \} $$
where $L^\vee := \Hom(R,L)$. 

\bk

In particular, for $m=2$, we set $X_0$ to be the quadratic $R$-algebra $R \oplus R$ 
with the standard involution $(r_1,r_2) \mapsto (r_1,-r_2)$, thus $\textbf{Aut}_{X_0} = \un{\mu}_2$.  
%with respect to the Galois action $(r,r) \mapsto (r,-r)$.  
Let $L$ be a fractional ideal of order $2$ in $\Pic(R)$.    
Then any $2$-degree Kummer pair $\Lm = (L,h)$ 
gives rise to an $R$-algebra $X= R \oplus L$ 
with multiplication defined by $(0,l_1) \cdot (0,l_2) = (h^{-1}(l_1 \otimes l_2),0)$, viewed as an $R$-form, 
whence according to Proposition \ref{flat classification}, $\Lm$ corresponds to the $\un{\mu}_2$-torsor 
$$ E_\Lm = \textbf{Iso}(R \oplus R, R \oplus L), $$ 
in which the isomorphism induced by $h$ is $\p:(r_1,r_2) \mapsto (r_1,l)$ where $l$ is such that $l \otimes l = h(r_2)$,   
i.e., $\p \otimes \p = h$.

\bk

Let $\OV$ be the \emph{orthogonal group} of $(V,q)$, namely, the $\CS$-group    
assigning to any $\CS$-algebra $R$ the group of self isometries of $q$ defined over $R$: 
$$ \OV(R) = \{ A \in \textbf{GL}(V \otimes R): \ q \circ A = q \}. $$ 
The pointed-set $H^1_\fl(\CS,\OV)$ classifies the integral quadratic forms of rank $n$ (cf. \cite[IV.5.3.1]{Knu}), 
being all locally isomorphic for the flat topology to $(V,q)$.  
We may generalize this to $\G$-forms as follows: 
Suppose $(V,q)$ is a $\G$-form. 
Then restricting $\p$ in Proposition \ref{flat classification}    
to the $\CS$-group of $\G$-automorphisms $\OV^\G$ for the flat topology, 
the corresponding forms are the integral $\G$-forms 
that are locally $\G$-isomorphic to $(V,q)$ in the flat topology. 
Modulo $\CS$-isomorphisms we get $H^1_\fl(\CS,\OV^\G)$.  

\bk

As $2 \in \CS^\times$ and $(V,q)$ is $\CS$-regular, $\OV$ is smooth %(cf. \cite[Theorem~1.7]{Con}). 
and its connected component, the \emph{special orthogonal group} of $(V,q)$,  
is $\SOV=\ker[\OV \stackrel{\det}{\rightarrow} \un{\mu}_2]$, containing only the proper isometries  
(cf. \cite[Thm.~1.7,~Cor.~2.5]{Con}).  
The push-forward homomorphism $\det_*:H^1_\fl(\CS,\OV) \to H^1_\fl(\CS,\un{\mu}_2)$ induced by flat cohomology,   
assigns to any quadratic form $(V',q')$ (taken up to an $\CS$-isomorphism) 
the class of its \emph{discriminant module} $D(q') = D(V',q') = (\wedge^n V',\det(q'))$ (see \cite[IV,~4.6,~5.3.1]{Knu}).   
Let $\SOV^\G$ be its $\G$-invariant subgroup. 
Our assumption that $|\G| \in \CS^\times$ guarantees that $\SOV^\G$ 
remains smooth (cf. \cite[Proposition~A.8.10(2)]{CGP2}). 
Any representative $(V',q')$ of a class in $H^1_\fl(\CS,\SOV^\G)$ represents a class in $H^1_\fl(\CS,\OV^\G)$, 
though $H^1_\fl(\CS,\SOV^\G)$ should not embed in $H^1_\fl(\CS,\OV^\G)$ 
(notice that these pointed-sets do not have to be groups).  
More precisely, any class in $H^1_\fl(\CS,\SOV^\G)$ 
is represented by a triple $(V',q',\theta')$ 
where $\theta'$ is the trivialization of $D(q')$, 
namely, an isomorphism $\theta': D(q') \otimes S \cong D(q)$ where $S$ is some $2$-degree flat 
%(also \'etale as $2 \in \CS^\times$) 
extension of $\CS$. 
Any proper $\G$-isometry $A: V' \cong V: q \circ A = q'$  
induces an isomorphism $D(A) : D(q') \otimes S \cong D(q)$. 
The question is whether these additional data $D(q')$ and $D(A)$ %in case $A$ is $\G$-invariant  
(when $q'$ is a $\G$-form) are also $\G$-equivariant. 

\bk

In order to answer this question we may consider as mentioned before the $\CS$-module $D(q')$, 
being of rank $1$ and isomorphic to $\CS$ over some at most $2$-degree \'etale extension,   
as a $2$-degree Kummer pair $(L=D(q'),h)$,   
giving rise to the $\un{\mu}_2$-torsor $E_{q'}$ for which  
$E_{q'}(\CS) = \{ \p \in L^\vee: \p \otimes \p = h \}$. % (see Section \ref{Kummer pairs}).  
Since ${q'}$ is a $\G$-form, $\G$ admits the following commutative diagram: 
\begin{equation*}
\xymatrix{
\G \ \ar@{^{(}->}[r]^{\rho} \ar[rd] & \un{\textbf{O}}_{V'}(\CS) \ar[d]^{\det}  \\
                                    & \un{\mu}_2(\CS) 
}
\end{equation*} 
from which we see that $\G$ acts on $E_{q'}(\CS)$ through its determinant in $\un{\mu}_2(\CS)$. 
But $\un{\mu}_2(\CS)$ is the automorphism group of $\CS \oplus \CS$ 
with respect to its standard involution $\tau=(\text{id},-\text{id})$  
and $E_{q'}(\CS)$ is stable (not fixed point-wise) under $\tau$, as $-l \otimes -l = l \otimes l$.  
Furthermore, the correspondence $A \leadsto D(A)$ is functorial,  
thus referring to any $\g \in \G$ as to an isometry one has 
$$ \g A \g^{-1} = A \ \Longleftrightarrow \ D(\g) D(A) D(\g)^{-1} = D(A), $$
i.e. $D(A)$ is $\G$-invariant as well. 
So Proposition \ref{flat classification} is also applied to %$H^1_\fl(\CS,\SOV^\G)$, 
the proper classification. 

\begin{cor} \label{G genus}
Given an integral $\G$-form base-point $(V,q)$, 
the pointed set $H^1_\fl(\CS,\OV^\G)$ $\mathrm{[}H^1_\fl(\CS,\SOV^\G)\mathrm{]}$  
$\mathrm{[}$properly$\mathrm{]}$ classifies the integral $\G$-forms that are locally 
$\G$-isomorphic to $(V,q)$ for the flat topology. 
\end{cor}

\bk

\section{The [proper] genus and [proper] $\G$-genus} \label{genus}
%-----------------------------------------------------
Consider the ring of $S$-integral ad\`eles $\A_S := \prod_{\fp \in S} \hat{K}_\fp \times \prod_{\fp \notin S} \hat{\CO}_\fp$, 
being a subring of the ad\`eles $\A$.  
The $S$-\emph{class set} of an affine and flat $\CS$-group $\un{G}$ is the set of double cosets:  
$$ \ClS(\un{G}) := \un{G}(\A_S) \backslash \un{G}(\A) / G(K) $$
(where for any prime $\fp$ the geometric fiber $\un{G}_\fp$ of $\un{G}$ is taken).    
According to Nisnevich (\cite[Theorem~I.3.5]{Nis}) $\un{G}$ admits the following exact sequence of pointed sets  
\begin{equation} \label{Nis sequence}
1 \to \ClS(\un{G}) \to H^1_\fl(\CS,\un{G}) \to H^1(K,G) \times \prod_{\fp \notin S} H^1_\fl(\hat{\CO}_\fp,\un{G}_\fp)     
\end{equation} 
in which the left exactness reflects the fact that $\ClS(\un{G})$ is the \emph{genus} of $\un{G}$, 
namely the set of (classes of) $\un{G}$-torsors that are generically and locally at $\fp \notin S$ isomorphic to $\un{G}$.  
If $\un{G}$ admits the property
\begin{equation} \label{locally embedded}
\forall \fp \notin S : \ \ H^1_\fl(\hat{\CO}_\fp,\un{G}_\fp) \hookrightarrow H^1_\fl(\hat{K}_\fp,G_\fp),  
\end{equation}  
then due to Corollary 3.6 in \cite{Nis} the Nisnevich's sequence \eqref{Nis sequence} simplifies to  %(cf. \cite[Corollary A.8]{GP}):
\begin{equation} \label{Nis simple}
1 \to \ClS(\un{G}) \to H^1_\fl(\CS,\un{G}) \to H^1(K,G),       
\end{equation} 
which indicates that any two $\un{G}$-torsors belong to the same genus if and only if they are $K$-isomorphic.

\begin{rem} \label{finite etale extension is embedded in generic fiber} 
Since $\Sp \CS$ is normal, i.e., is integrally closed locally everywhere (due to the smoothness of $C$), 
any finite \'etale covering of $\CS$ arises by its normalization %of $\Sp \CS$ 
in some separable unramified extension of $K$ (see \cite[Theorem~6.13]{Len}). 
Consequently, if $\un{G}$ is a finite $\CS$-group, 
then $H^1_\et(\CS,\un{G})$ is embedded in $H^1(K,\un{G})$. 
\end{rem}

\begin{lem} \label{genus of subgroup}
Let $\p:\un{G} \to \un{G}'$ be a monomorphism of smooth affine $\CS$-groups   
and let $\wt{\un{Q}}$ be the sheafification of $\un{Q}:=\cok(\p)$. Then  
\begin{itemize}{}
\itemindent=0em
\item The map $\un{G}'(\CS) \to \wt{\un{Q}}(\CS)$ is surjective iff   
$ \ \ker[\ClS(\un{G}) \stackrel{\psi}{\rightarrow} \ClS(\un{G}')]=~1$.  
\item If, moreover, $\un{G}$ is locally of finite presentation, $\un{G}$ and $\un{G}'$ admit property \eqref{locally embedded}, 
$\un{Q}$ is a finite $\CS$-group and $G'(K) \to Q(K)$ is surjective, then $\psi$ is surjective. 
\end{itemize} 
\end{lem}

\begin{proof} 
As a pointed set, $\ClS(\un{G})$ is bijective to the first Nisnevich's cohomology set 
$H^1_{\text{Nis}}(\CS,\un{G})$ (cf. \cite[I.~Theorem~2.8]{Nis}), 
classifying $\un{G}$-torsors for the Nisnevich's topology. 
But Nisnevich's covers are \'etale, so $\ClS(\un{G})$ is a subset of $H^1_\et(\CS,\un{G})$. 
The monomorphism $\p$ does not have to be a closed immersion hence $\un{Q}$ may not be representable, 
%though being formally smooth and locally of finite presentation over $\CS$, 
%thus smooth (see Definition \ref{formally smooth}).  
and so we may not be able to apply flat cohomology on the obtained short exact sequence of $\CS$-schemes.   
Restricting, however, to the small site of flat extensions of $\CS$, 
we have $\wt{\un{G}}(R) \subseteq \wt{\un{G}'}(R)$ for any such extension $R/\CS$, 
where $\wt{\un{G}}$ and $\wt{\un{G}'}$ stand for the sheafifications of $\un{G}$ and $\un{G}'$, respectively. 
Then flat cohomology applied to the exact sequence of flat sheaves 
\begin{equation} \label{sheaves} 
1 \to \wt{\un{G}} \xrightarrow{i} \wt{\un{G}'} \to \wt{\un{Q}} \to 1
\end{equation}
yields a long exact sequence in which $H^1_\fl(\CS,\wt{*}) = H^1_\fl(\CS,*) = H^1_\et(\CS,*)$ 
for both smooth groups $* = \un{G}'$ and $\un{G}$,   
whence $\wt{\un{G}'}(\CS) = \un{G}'(\CS) \to \wt{\un{Q}}(\CS)$ is surjective if and only if  
$\ker[H^1_\et(\CS,\un{G}) \xrightarrow{\psi} H^1_\et(\CS,\un{G}')]=1$,   
being equivalent by restriction to $\ClS(\un{G})$, 
to $\ker[\ClS(\un{G}) \xrightarrow{\psi} \ClS(\un{G}')]=1$, 
since any twisted form in $H^1_\et(\CS,\un{G})$ can be $\CS$-isomorphic to $\un{G}$ only if it belongs to $\ClS(\un{G})$.

Now suppose  $\un{G}$ is locally of finite presentation and $\un{Q}$ is representable as a finite $\CS$-group. 
Then $\un{Q}$ is smooth as well (cf. \cite[VI$_\text{B}$, Proposition~9.2~xii]{SGA3}) 
and so given furthermore that $\un{G}$ and $\un{G}'$ admit property \eqref{locally embedded} and  $G'(K) \to Q(K)$ is surjective,     
\'etale cohomology applied to sequence \eqref{sheaves} over $\CS$ and over $K$  
extends the exactness of sequence \eqref{Nis simple} to the commutative diagram 
\begin{equation*}
\xymatrix{
         & \ClS(\un{G})        \ar[r]^{\psi} \ar@{^{(}->}[d]^{i} & \ClS(\un{G}')        \ar@{^{(}->}[d]^{i'}     \\
1 \ar[r] & H^1_\et(\CS,\un{G}) \ar[r]^{\psi} \ar[d]^{m}          & H^1_\et(\CS,\un{G}') \ar[d]^{m'} \ar[r]^{d}   & H^1_\et(\CS,\un{Q}) \ar@{^{(}->}[d]^{m''}  \\
1 \ar[r] & H^1(K,G)            \ar@{^{(}->}[r]^{h}               & H^1(K,G')            \ar[r]^{d'}              & H^1(K,Q) 
}
\end{equation*} 
in which $m''$ is injective as $\un{Q}$ is finite, by Remark \ref{finite etale extension is embedded in generic fiber}.  
We then get the surjectivity of $\psi$: 
\begin{align*}
x' \in \ClS(\un{G}') \ &\Rightarrow \ m''(d(i'(x'))) = d'(m'(i'(x'))=0)=0 \ \Rightarrow \ d(i'(x'))=0  \\ 
                       &\Rightarrow \ \exists y \in H^1_\et(\CS,\un{G}):\psi(y) = i'(x') \ \Rightarrow \  m'(\psi(y)) = h(m(y)) = 0 \\
											 &\Rightarrow \ m(y)=0 \ \Rightarrow \ \exists x \in \ClS(\un{G}):\psi(i(x)=y) = i'(x') \ \Rightarrow \ \psi(x) = x'. %\ \ \qedhere  
\end{align*}  
\end{proof} 

We return now to our integral $\G$-form $(V,q)$ for which $\ClS(\OV^\G) = c_\G(q)$ and $\ClS((\OV^\G)^0) = c_\G^+(q)$. 
Since $\OV^\G$ and $(\OV^\G)^0$ are smooth, their flat cohomology sets over $\Sp \CS$ 
coincide with the \'etale ones (cf. \cite[VIII~Corollaire~2.3]{SGA4}). 
According to Lemma \ref{genus of subgroup}, if $\OV(\CS) \to (\wt{\OV/\OV^\G})(\CS)$ is surjective 
then $\ker[c_\G(q) \xrightarrow{\psi} c(q)]=1$, 
which means that for any $[q'] \in c_\G(q)$, 
if $q'$ is not $\G$-isomorphic to $q$ then neither is it $\CS$-isomorphic to it. 
This still does not imply that $c_\G(q)$ injects into $c(q)$ since both are not necessarily groups, 
so there may be two classes of forms other than $[q]$, 
which are distinct in $c_\G(q)$, yet are $\CS$-isomorphic. 
We summarize by 

\begin{prop} \label{all genus} 
%Let $[(V',q')] \in C_\G(q)$ be the orthogonal group of a twisted form $\un{Q}:=\OV / \OV^\G$. 
Question \ref{question} is answered in the affirmative 
if and only if $\un{\textbf{O}}_{V'}(\CS) \to (\wt{\un{\textbf{O}}_{V'}/\un{\textbf{O}}_{V'}^\G})(\CS)$ is surjective for any $[(V',q')] \in c(q)$. 
\end{prop}

\begin{lem} \label{connected component reductive}
The group scheme $(\OV^\G)^0$ is reductive. 
\end{lem}

\begin{proof}
The group scheme $\SOV$ is reductive as it is smooth, affine and all its fibers are reductive. 
As mentioned before, its affine subgroup $\SOV^\G$ is smooth as well, 
so we may refer to its neutral component $(\SOV^\G)^0$ as defined in \cite[VIB, Th\'eor\`eme~3.10]{SGA3}.  
The reduction $(\ov{\textbf{SO}}_V)_\fp$ defined over the residue field $k_\fp$ at any prime $\fp$ is reductive, 
hence according to \cite[Proposition~A.8.12]{CGP2} its subgroup $((\ov{\textbf{SO}}_V)_\fp^\G)^0 = (\ov{\textbf{SO}}_V^\G)^0_\fp$  
remains reductive, so that $(\SOV^\G)^0$ is reductive. 
The latter being a smooth, open and connected subgroup of $\OV^\G$ coincides with $(\OV^\G)^0$, thus it is reductive.
%The connected component $\SOV$ is reductive in the fiberwise sense, 
%i.e., its reduction $(\ov{\textbf{SO}}_V)_\fp$ at any prime $\fp$ is reductive over the residue fie _\fp$. 
%Thus also $((\ov{\textbf{SO}}_V)_\fp^\G)^0$ is reductive (cf. \cite[Prop.~A.8.12]{CGP2}), being equal to $((\ov{\textbf{SO}}_V^\G)_\fp)^0=((\ov{\textbf{O}}_V^\G)_\fp)^0$  
%which is the reduction of $(\OV^\G)^0$ at $\fp$. 
\end{proof}

\begin{rem} \label{admit property} 
The group $(\OV^\G)^0$ admits property \eqref{locally embedded} 
by Lang's Theorem (recall that all residue fields are finite).   
Being reductive over $\Sp \CS$ (Lemma~\ref{connected component reductive}),   
this property holds for $\OV^\G$ as well if $\OV^\G/(\OV^\G)^0$ is representable as a finite $\CS$-group  
(see the proof of Proposition~3.14 in \cite{CGP1}). 
\end{rem}

Following sequence \eqref{Nis simple} and Remark \ref{admit property} we then get: 

\begin{cor} \label{genus as kernel} 
$c^+_\G(q) \cong \ker[H^1_\et(\CS,(\OV^\G)^0) \to H^1(K,(\textbf{O}_V^\G)^0)]$. \\
If $\OV^\G/(\OV^\G)^0$ is a finite $\CS$-group then $c_\G(q) \cong \ker[H^1_\et(\CS,\OV^\G) \to H^1(K,\textbf{O}_V^\G)]$.  
\end{cor}  

\begin{cor} \label{genus and proper genus}
If $(V,q)$ of rank $\geq 3$ is regular and isotropic then $c^+(q)=c(q)$. 
\end{cor}

\begin{proof}
Any representative $(V',q')$ of a class in $c(q)$ is $\CS$-regular,  
thus $\un{\textbf{O}}_{V'}/\un{\textbf{SO}}_{V'} \cong \un{\mu}_2$ (cf. \cite[Thm.~1.7]{Con}, 
notice that %$(\Z/2\Z)_{\CS} 
$\un{\Z/2}:= \Sp \Z[t]/(t(t-1))$ is isomorphic to $\un{\mu}_2$ as $2 \in \CS^\times$).     
Moreover, being $K$-isomorphic to $q$ (due to Corollary \ref{genus as kernel}), 
$q'$ is isotropic as well, therefore as $\text{rank}(V) \geq 3$, the map $\un{\textbf{O}}_{V'}(\CS) \xrightarrow{\det} \un{\mu}_2(\CS) = \mu_2(K) = \{\pm 1\}$ is surjective (cf. \cite[Lemma~4.3]{Bit2}).  
So setting $\un{G} = \un{\textbf{SO}}_{V'}$ and $\un{G}' = \un{\textbf{O}}_{V'}$ in Lemma \ref{genus of subgroup}, 
we get that $\ker[c^+(q') \xrightarrow{\psi} c(q')=c(q)]=1$ and $\psi$ is surjective. 
This holds as mentioned before to any $[q'] \in c(q)$, which amounts to $\psi$ being the identity.     
\end{proof}

%\begin{proof}
%Any representative $(V',q')$ of a class in $c(q)$, being $K$ isomorphic to $q$ (due to Corollary \ref{genus as kernel}), 
%is regular and isotropic, thus $\un{\textbf{O}}_{V'}/\un{\textbf{SO}}_{V'} \cong \un{\mu}_2$ (cf. \cite[Theorem~1.7]{Con}, 
%notice that %$(\Z/2\Z)_{\CS} 
%$\un{\Z/2}:= \Sp \Z[t]/(t(t-1))$ is isomorphic to $\un{\mu}_2$ as $2 \in \CS^\times$),   
%and as $\text{rank}(V) \geq 3$, $\un{\textbf{O}}_{V'}(\CS) \xrightarrow{\det} \un{\mu}_2(\CS) = \mu_2(K) = \{\pm 1\}$ is surjective (cf. \cite[Lemma~4.3]{Bit2}).  
%So setting $\un{G} = \un{\textbf{SO}}_{V'}$ and $\un{G}' = \un{\textbf{O}}_{V'}$ in Lemma \ref{genus of subgroup}, 
%we get that $\ker[c^+(q') \xrightarrow{\psi} c(q')=c(q)]=1$ and $\psi$ is surjective. 
%This holds as mentioned before to any $[q'] \in c(q)$, which amounts to $\psi$ being the identity.     
%\end{proof}

\begin{rem} \label{anisotropic}
When $|S|=1$, i.e., $S = \{\iy\}$ where $\iy$ is an arbitrary closed point of $K$, Lemma \ref{genus and proper genus} is automatic for any regular $q$ of rank $\geq 3$,  
%the equality $c^+(q)=c(q)$ is automatic; 
%indeed, for $\text{rank}(V) = 1$ this is trivial. 
%As $c^+(q)=H^1_\et(\Oiy,\SOV)$ %and $c(q)=H^1_\et(\Oiy,\OV)$ 
%(see the proof of \cite[Prop.~4.2]{Bit1}), 
%for $\text{rank}(V) = 2$, the space $(V'=\Oiy^2,q'=\la 1,\diag(\det(B_q))\ra)$ (with the standard basis) belongs to $c^+(q)$. %represented by $B_{q'} = \diag(1,\det(q))$ 
%It admits the self non-proper isometry $\diag(1,-1)$ even without being isotropic,  
%so $\det(\un{\textbf{O}}_{V'}(\Oiy))=\{\pm 1\}$ and consequently $\ker[H^1_\et(\Oiy,\un{\textbf{SO}}_{V'}) \to H^1_\et(\Oiy,\un{\textbf{O}}_{V'})]=1$. 
%But $\un{\textbf{SO}}_{V'}$, being smooth, connected and one dimensional is a torus, i.e. commutative, thus $c^+(q)$, being equal to $H^1_\et(\Oiy,\un{\textbf{SO}}_{V'})$ 
%up to the choice of a base-point, is an abelian group, and so $c^+(q) \xrightarrow{\psi} c(q)$ is injective. 
%It is also surjective due to Lemma \ref{genus of subgroup} (by setting $G=\SOV$ and $G'=\OV$) and Remark \ref{admit property} as $\OV \cong \SOV \rtimes \un{\mu}_2$.  
since in that case $(V,q)$ must be isotropic (see the proof of \cite[Prop.~4.4]{Bit1}). % and the proof is completed by Corollary \ref{genus and proper genus}. 
\end{rem}

\begin{lem} \label{semi-direct} %{same forms genus} 
Suppose that $\OV^\G \cong (\OV^\G)^0 \rtimes \un{Q}$ where $\un{Q}$ is a finite $\CS$-group.  \\
Then $\ker[c_\G^+(q) \xrightarrow{\psi} c_\G(q)]=1$ and $\psi$ is surjective.  
\end{lem}

\begin{proof} 
As $\un{Q}$ embeds as a semi-direct factor in $\OV^\G$, $\OV^\G(\CS)$ surjects onto $\un{Q}(\CS)$  
%the obtained exact sequence of group sheaves 
%$$ 1 \to \wt{(\OV^\G)^0} \to \wt{\OV^\G} \to \wt{\un{Q}} \to 1 $$
%splits, i.e., there is a section $\wt{\un{Q}} \to \wt{\OV^\G}$, 
%implying that $\OV^\G(\CS) \to \un{Q}(\CS)$ 
and $\textbf{O}_V^\G(K)$ onto $Q(K)$. % are surjective. 
Given furthermore that $\un{Q}$ is a finite $\CS$-group, 
both $(\OV^\G)^0$ and $\OV^\G$ admit property \eqref{locally embedded} (see Remark \ref{admit property}),   
so all conditions in Lemma \ref{genus of subgroup} are satisfied and the assertion follows.  
\end{proof}

\begin{lem} \label{rank 2}
If $\text{rank}(V) = 2$ then $c_\G(q) \subseteq c(q)$,   
i.e., Question \ref{question} is then answered positively. 
\end{lem}   

\begin{proof}
If $\text{rank}(V)=2$ then 
%$\OV \cong \SOV \rtimes \un{\mu}_2$ where 
$\SOV$ is a one dimensional torus. 
We first claim that $\OV^\G$ cannot be $\SOV$. 
Indeed, since $q$ is a $\G$-form, $\G$ embeds by definition in $\OV(\CS)$. 
%If $\G$ is trivial then also $\OV / \SOV = \un{\mu_2}$ is $\G$-invariant. 
%If $\G$ has a non-trivial image in $\un{\mu}_2(\CS)$ then it cannot stabilize $\OV(\CS)$, 
As $(\OV / \SOV)(\CS) = \un{\mu}_2(\CS)$ does not commute with $\OV(\CS)$, 
$\G$ cannot have a non-trivial image in it and still stabilizing $\SOV$,   
i.e., it must embed in $\SOV(\CS)$ only. 
But then being finite, the $\G$-image must be the group $\{\pm I_2\}$, which does not kill $\un{\mu}_2(\CS)$. 
So either $\OV^\G=\OV$ for which the assertion is trivial, or $\OV^\G$ is a finite group, 
for which $c_\G(q) \stackrel{\eqref{genus as kernel}}{\cong} \ker[H^1_\et(\CS,\OV^\G) \to H^1(K,\textbf{O}_V^\G)]$   
is trivial according to Remark \ref{finite etale extension is embedded in generic fiber}.  
\end{proof}

As mentioned in the proof of Lemma \ref{rank 2}, 
the inequality $c_\G(q) \subsetneq c(q)$ for $\text{rank}(V)=2$ 
may occur only when $\OV^\G$ is finite.  
For example, suppose $\OV^\G = \un{\mu}_m$.   
Then \'etale cohomology applied to the related Kummer exact sequence of smooth $\CS$-groups  
$$ 1 \to \un{\mu}_m \to \un{\BG}_m \to \un{\BG}_m \to 1 $$ 
yields the exactness of 
\begin{equation} \label{LHS H1 mum}
1 \to \CS^\times/(\CS^\times)^m \to H^1_\et(\CS,\un{\mu}_m) \to {_m}\Pic(\CS) \to 1 
\end{equation}
where the right non-trivial term stands for the $m$-torsion part of $\Pic(\CS)$. 
According to Remark \ref{finite etale extension is embedded in generic fiber}, since $\Sp \CS$ is normal, %this is also true %for the geometrical part, namely, 
both $\CS^\times/(\CS^\times)^m$ and the preimage of ${_m}\Pic(\CS)$ in $H^1_\et(\CS,\un{\mu}_m)$ are embedded in $K^\times /(K^\times)^m \cong H^1(K,\mu_m)$. 
% is embedded in its generic fiber $H^1(K,\mu_m)$ (see ). 
The following example demonstrates this embedding, which yields the above inequality.

\begin{exa}
Consider the elliptic curve $C = \{ Y^2Z = X^3 + XZ^2\}$ defined over $\BF_{11}$. 
Removing the closed point $\iy = (0:1:0)$ results in an affine curve with coordinate ring:  
$$ C^\af = \{y^2 = x^3 + x \} \ \ \text{and:} \ \ \CS = \BF_{11}[C^\af].  $$
Let $V = \CS^2$ be generated by the standard basis over $\CS$ endowed with the form $q$ represented by $B_q=1_2$.   
Then 
$$ \SOV = \un{\textbf{SO}}_2 =  
\left \{
\left( 
\begin{array}{cc} 
  x & -y \\       
  y &  x  
\end{array}
\right): x^2+y^2=1 \right \} $$ 
is a one dimensional $\CS$-torus and as $-1$ is not a square, 
it is isomorphic to the non-split norm torus $\N := R^{(1)}_{\CS(i)/\CS}(\un{\BG}_m)$, 
fitting into the exact sequence of $\CS$-tori:
$$ 1 \to \N \to \un{R}:= R_{\CS(i)/\CS}(\un{\BG}_m) \xrightarrow{\det} \un{\BG}_m \to 1. $$ 
Then since $\un{R}(\CS) \xrightarrow{\det} \CS^\times = \BF_{11}^\times$ is surjective ($x^2+y^2$ gets any value in $\BF_{11}^\times$), 
\'etale cohomology together with the Shapiro's Lemma gives rise to the exact sequence: 
$$  1 \to H^1_\et(\CS,\N) \to \Pic(\CS(i)) \to \Pic(\CS) $$
from which we see that $c^+(q) = \ClS(\SOV) = H^1_\et(\CS,\SOV \cong \N)$ (cf. \cite[Prop.~4.2]{Bit1}) is far from being trivial 
(say, by the Hasse-Weil bound: 
$|\Pic(\CS)| = |C^\af(\BF_{11})| < 11 +1 + 2 \sqrt{11} < 19$ 
while: 
$|\Pic(\CS(i))| = |C^\af(\BF_{11}(i))| > 121 + 1 - 2 \sqrt{121} = 100$, see in Example \ref{last example}). 
%As $\OV \cong \SOV \rtimes \un{\mu}_2$ %by Lemma \ref{semi-direct}  
As $\OV / \SOV \cong \un{\mu}_2$ and $\OV(\CS) \xrightarrow{\det} \un{\mu}_2(\CS)$ is surjective by $\diag(1,-1) \mapsto -1$, 
setting $\un{G}=\SOV$ and $\un{G}'=\OV$ in Lemma \ref{genus of subgroup} 
we get that $c(q) = c^+(q)$, thus is not trivial as well.  

Now let $\G = S_3 =\la \tau,\s \ra$ be represented in $V$ by:
$ \tau \mapsto   
\left( 
\begin{array}{cc} 
  0 & 1 \\       
  1 & 0  
\end{array}
\right)$ and:  
$\s \mapsto 
\left( 
\begin{array}{cc} 
  5 & -8 \\       
  8 & 5  
\end{array}
\right)$. 
One can easily check that $q$ is a $\G$-form and that  
$$
\SOV^\G =  
\left \{
\left( 
\begin{array}{cc} 
  x & 0 \\       
  0 & x  
\end{array}
\right): x^2=1 \right \} \cong \un{\mu}_2. $$
For $L = \la x,y \ra \in \Pic(\CS)$ one has: 
$$L \otimes L = \la x^2,xy,y^2 \ra = \la x^2,xy,x^3+x \ra \subseteq \la x \ra. $$
But $x=y^2-x^3 \in L \otimes L$, thus $L \otimes L = \la x \ra$.   
So the $2$-degree Kummer pair $(L,h)$ gives rise to the %degree $2$ covering $\CS \oplus L \to \CS: (r_1,l_1) = (r_1 \cdot l_1)/x$, 
$\un{\mu}_2$-torsor $\CS \oplus L$ being isomorphic to $\CS^2$ over $\CS[1/\sqrt{x}]$ which is not contained in $K$. 
%This means that the corresponding $\un{\mu}_2$-torsor is embedded in $H^1(K,\mu_2)$.  
The same happens for the other $\un{\mu}_2$-torsors, 
i.e., $c_\G(q) \stackrel{\eqref{genus as kernel}}{\cong} \ker[H^1_\et(\CS,\un{\mu}_2) \to H^1(K,\mu_2)]=1 \subsetneq c(q)$. 
%$$ \ClS(\SOV^\G)= \Ker[H^1_\et(\CS,\un{\mu}_2) \to H^1(K,\mu_2) \times H^1_{(2)}(\hat{\CO}_{(2)},\un{\mu}_2)] $$ is trivial.  
\end{exa}

\bk

\section{An explicit obstruction} \label{Obstruction}
%--------------------------------
The criterion exhibited in Proposition \ref{all genus} for Question \ref{question} 
to be answered in the affirmative, namely, $\un{\textbf{O}}_{V'}(\CS) \to (\wt{\un{\textbf{O}}_{V'}/\un{\textbf{O}}_{V'}^\G})(\CS)$ 
is surjective for any $[(V',q')] \in c(q)$, is somewhat vague. 
We would like to refer to the case in which  $(\OV^\G)^0$ 
is the special orthogonal group of another isotropic $\G$-form.  
%In this case this question may be answered negatively as follows: 
It is shown in \cite[Proposition~4.4]{Bit1} for $|S|=1$ and more generally in \cite[Theorem~4.6]{Bit2} for any finite $S$, 
that if $\text{rank}(V) \geq 3$ ($q$ is isotropic), then $c(q) \cong \Pic(\CS)/2$. 
For $\text{rank}(V) =2$, however, this genus might be larger. 
This means that there may be two integral forms %$(V',q')$ and $(V'',q'')$ 
of rank $2$ that are only \emph{stably isomorphic}, 
i.e., become isomorphic after being extended by any non-trivial regular common extension. 
This failure of the Witt Cancellation Theorem over $\CS$, 
invokes a case in which Question \ref{question} is answered negatively,     
namely, when $\text{rank}(V) \geq 3$ and $(\OV^\G)^0$ is the special orthogonal group of another integral $\G$-form $(V',q')$ of rank $2$, 
whose genus decreases over $V$.    

\begin{prop} \label{counter example}
Let $(V,q)$ be a regular $\G$-form of rank $\geq 3$ such that $(\OV^\G)^0$ 
is the special orthogonal group of an isotropic form of rank $2$,  
being a semi-direct factor in $\OV^\G$, while the quotient is a finite $\CS$-group.  
If $-1 \in (\BF_q^\times)^2$ and $\text{exp}(\Pic(\CS))>2$ then $c_\G(q)$  does not inject into $c(q)$, 
i.e., Question \ref{question} is then answered negatively.   
\end{prop}

\begin{proof}
Given that $\text{rank}(V') = 2$ and $-1 \in (\BF_q^\times)^2$, 
$\un{\textbf{O}}_{V'}^0 \cong \un{\BG}_m$, so one has: 
$$ 
c_\G^+(q) = c^+(q') \stackrel{\eqref{genus as kernel}}{\cong} \ker[H^1_\et(\CS,\un{\BG}_m) \to H^1(K,\BG_m)]. 
$$
This kernel is isomorphic due to Shapiro's Lemma and Hilbert's 90 Theorem to $\Pic(\CS)$.  
Since $(\OV^\G)^0$ is a normal semi-direct factor in $\OV^\G$ and the quotient is a finite $\CS$-group, 
by Lemma \ref{semi-direct} $\ker[c_\G^+(q) \xrightarrow{\psi} c_\G(q)]=1$, 
so $\ker[c_\G(q) \to c(q)]$ cannot vanish, because if it would, 
the composition, being a morphism of abelian groups
$$ c_\G^+(q) \cong \Pic(\CS) \to c(q) \cong \Pic(\CS)/2 $$
would be injective, which is impossible whenever $\text{exp}(\Pic(\CS))>2$. 
\end{proof}

\begin{exa} \label{last example}
Let $C$ be an elliptic $\BF_q$-curve such that $-1 \in (\BF_q^\times)^2$ and $\text{exp}(C(\BF_q))>2$.  
Let $\iy$ be an $\BF_q$-rational point. 
Then $\Pic(\CS) \cong C(\BF_q)$ (cf. e.g., \cite[Example~4.8]{Bit1}). 
Let $(V,q)$ be the quadratic space generated by the standard basis and represented by $B_q=1_n$ for $n \geq 4$. 
Then as mentioned before, $c(q) \cong \Pic(\CS)/2 \cong C(\BF_q)/2$.  
Let the permutations in $\G = S_{n-2}$ be canonically represented by monomial matrices 
in the lower right $(n-2) \times (n-2)$ block of $\OV \subset \un{\textbf{GL}}(V)$:
$$ 
\G \hookrightarrow 
\left( 
\begin{array}{ccc} 
  1 & 0 & 0 \\       
  0 & 1 & 0 \\
	0 & 0 & S_{n-2}
\end{array}
\right)
$$
turning $q$ into a $\G$-form. 
Now if $n$ is even, then 
%$$ 
$\SOV^\G \cong \un{\BG}_m \times \{\pm I_{n-2}\}$. 
Otherwise, if $n$ is odd, then $\SOV^\G$ is a semi-direct product of $\un{\BG}_m \times \{I_{n-2}\} \cong ~\un{\BG}_m$ 
and $\diag(1,-1) \times \{-I_{n-2}\} \cong \un{\mu}_2$.  
%\cong \un{\BG}_m \times \un{\mu}_2 & n \ \text{even} 
%\end{array} \right. ,  $$ 
In both cases, $(\OV^\G)^0 = (\SOV^\G)^0 \cong \un{\BG}_m$ 
is a normal semi-direct factor in $\OV^\G$ and the quotient is a finite $\CS$-group, 
therefore according to Proposition \ref{counter example} $c_\G(q)$ cannot inject into $c(q)$.  
%so Question \eqref{question} is answered negatively, by any integral $\Gamma$-form corresponding to a fractional ideal of $\CS$ of order $>2$.     
\end{exa}

\bk

\end{document}